\documentclass[a4paper,12pt]{article}
\usepackage[top=2.5cm,bottom=2.5cm,left=2.5cm,right=2.5cm]{geometry}
\usepackage{cite,amsmath,amssymb}
    \usepackage[margin=1cm,%
                font=small,%
                format=hang,%
                labelsep=period,%
                labelfont=bf]{caption}
\usepackage[algoruled,linesnumbered]{algorithm2e}
\usepackage{algorithmic}

\usepackage{amsfonts,epsf,here,graphicx}
\usepackage{latexsym,multirow,rotating}
\usepackage{pgf,tikz,subfigure}
\usepackage{epstopdf}

\newtheorem{theorem}{\bf Theorem}[section]
\newtheorem{corollary}[theorem]{\bf Corollary}

\newcommand{\qed}{\hfill $\square$ \bigskip}

\begin{document}

\baselineskip=0.30in
\vspace*{40mm}

\begin{center}
{\LARGE \bf The Edge-Wiener Index and the Edge-Hyper-Wiener Index of Phenylenes}
\bigskip \bigskip

{\large \bf Petra \v Zigert Pleter\v sek
}
\bigskip\bigskip

\baselineskip=0.20in

\textit{Faculty of Chemistry and Chemical Engineering, University of Maribor, Slovenia \\
Faculty of Natural Sciences and Mathematics, University of Maribor, Slovenia} \\
{\tt petra.zigert@um.si}
\medskip

\bigskip\medskip

(\today)

\end{center}

\noindent
\begin{center} {\bf Abstract} \end{center}

\vspace{3mm}\noindent
Besides the well known Wiener index, which sums up the distances between all the pairs of vertices, and the hyper-Wiener index, which includes also the squares of distances, the edge versions of both indices attracted a lot of attention in the recent years. 

In this paper we consider the edge-Wiener index and the edge-hyper-Wiener index of phenylenes, which represent an important class of molecular graphs. For an arbitrary phenylene, four quotient trees based on the elementary cuts are defined in a similar way as it was previously done for benzenoid systems. The computation of the edge-Wiener index of the phenylene is then reduced to the calculation of the weighted Wiener indices of the corresponding quotient trees. Furthermore, a method for computing the edge-hyper-Wiener index of phenylenes is described. Finally, the application of these results gives closed formulas for the edge-Wiener index and the edge-hyper-Wiener index of linear phenylenes.

\baselineskip=0.30in



\section{Introduction}

The Wiener index of a graph is defined as the sum of distances between all pairs of vertices in the graph. It was introduced in 1947 by Wiener \cite{Wiener} and represents one of the most studied molecular descriptors. On the other hand, the hyper-Wiener index takes into account also the squares of distances and was first introduced by Randi\' c in 1993 \cite{randic}. The research on the Wiener index, the hyper-Wiener index and other distance-based descriptors is still a popular topic, see papers \cite{chen,dob,il-il,ramane} for some recent investigations.

The edge-Wiener index of a graph was introduced in \cite{iranmanesh-2009} as the Wiener index of the line graph and has been since then intensively investigated \cite{alizadeh-2014,azari-2011,chen2,dankelmann-2009,iranmanesh-2015,
knor-2014b,nadjafi-arani-2012}. However, the concept was studied even before as the Wiener index of line graphs, see \cite{buckley,gutman-1997b}. Similarly, the edge-hyper-Wiener index was introduced in \cite{edge-hyper}.

A cut method is a powerful method for efficient computation of topological indices of graphs. Some of the earliest results are related to the computation of the Wiener index of benzenoid systems and are presented in \cite{chepoi-1996,chepoi-1997,klavzar-1997}. See also \cite{klavzar-2015} for a survey paper on the cut method and \cite{aroc,cre-trat,cre-trat1,tratnik2} for some recent investigations on this topic. In \cite{kelenc} a cut method for the edge-Wiener index of benzenoid systems was proposed and in \cite{tratnik1} a cut method for the edge-hyper-Wiener index of partial cubes was developed.

Beside benzenoid hydrocarbons, phenylenes represent another interesting class of polycyclic conjugated molecules, whose properties have been extensively studied, see \cite{fur,gut-as}. The Wiener index and the hyper-Wiener index of phenylenes were studied in \cite{gu-do} and \cite{cash}, respectively. In this paper, we consider methods for computing the edge-Wiener index and the edge-hyper-Wiener index of phenylenes and use them to obtain closed formulas for linear phenylenes.

\section{Preliminaries}

The distance between two vertices $u,v$ of a graph, denoted by $d_G(u,v)$, is defined as the length of a shortest path between $u$ and $v$. The \textit{Wiener index} of a connected graph $G$ is

$$W(G) = \frac{1}{2} \sum_{u \in V(G)} \sum_{v \in V(G)} d_G(u,v)\,.$$

\noindent
To point out that it is the vertex-Wiener index, we will also write $W_v(G)$ for $W(G)$. From some technical reasons we also set $\widehat{d}_G(x,y) = d_G(x,y)$. The distance between two edges $e,f \in E(G)$, denoted by $d_G(e,f)$, is the usual shortest-path distance between vertices $e$ and $f$ of the line graph $L(G)$ of $G$. Here we follow this convention because in this way the pair $(E(G),d)$ forms a metric space. Then the \textit{edge-Wiener index} of a connected graph $G$ is defined as 

\begin{equation}
\label{eq:W-e-define}
W_e(G) = \frac{1}{2} \sum_{e \in E(G)} \sum_{f \in E(G)} d_G(e,f)\,.
\end{equation}

\noindent
In other words, $W_e(G)$ is just the Wiener index of the line graph of $G$. On the other hand, for edges $e = ab$ and $f = xy$ of a graph $G$ it is also legitimate to set
\begin{equation*}
\label{eq:hat-d}
\widehat{d}_G(e,f) = \min \lbrace d_G(a,x), d_G(a,y), d_G(b,x), d_G(b,y) \rbrace\,.
\end{equation*}

\noindent
Replacing $d$ with $\widehat{d}$ in~\eqref{eq:W-e-define}, another variant of the edge-Wiener index is obtained (see \cite{khalifeh-2009}) and we denote it by $\widehat{W}_e(G)$. However, there is an obvious connection between $W_e(G)$ and $\widehat{W}_e(G)$:

\begin{equation}
\label{eq:simple-connection}
\widehat{W}_e(G) =  W_e(G) - \binom{|E(G)|}{2}\,.
\end{equation}

\noindent
Finally, the {\em vertex-edge Wiener index} is
$$W_{ve}(G) = \sum_{x \in V(G)} \sum_{e \in E(G)} \widehat{d}_G(x,e)\,,$$

\noindent
where for a vertex $x \in V(G)$ and an edge $e=ab \in E(G)$ we set 
\begin{equation*}
\widehat{d}_G(x,e) = \min \lbrace d_G(x,a), d_G(x, b) \rbrace\,.
\end{equation*}

Next, we extend the above definitions to weighted graphs. Let $G$ be a connected graph and let $w:V(G)\rightarrow {\mathbb R}^+$  and $w':E(G)\rightarrow {\mathbb R}^+$ be given functions. Then $(G,w)$, $(G,w')$, and $(G,w,w')$ are a {\em vertex-weighted graph}, an {\em edge-weighted graph}, and a {\em vertex-edge weighted graph}, respectively. The corresponding Wiener indices of these weighted graphs are defined as
\begin{eqnarray*}
W(G,w) & = & \frac{1}{2} \sum_{x \in V(G)} \sum_{y \in V(G)} w(x)w(y)d_G(x,y)\,, \\
W_e(G,w') & = & \frac{1}{2} \sum_{e \in E(G)} \sum_{f \in E(G)} w'(e)w'(f)d_G(e,f)\,,\\
W_{ve}(G,w,w') & = & \sum_{x \in V(G)} \sum_{e \in E(G)} w(x)w'(e)\widehat{d}_G(x,e)\,. 
\end{eqnarray*}
Again, we will often use $W_v(G,w)$ for $W(G,w)$ and $\widehat{W}_e(G,w')$ is defined analogously as $W_e(G,w')$ by using $\widehat{d}_G(e,f)$ instead of $d_G(e,f)$.

\noindent
For an efficient computation of the Wiener indices of weighted trees we need some additional notation. If $T$ is a tree and $e \in E(T)$, then the graph $T-e$ consists of two components that will be denoted by  $C_1(e)$ and $C_2(e)$. For a vertex-edge weighted tree $(T,w,w')$, $e \in E(T)$, and $i \in \lbrace 1, 2 \rbrace$ set 
$$n_i(e) = \sum_{u \in V(C_i(e))}w_i(u) \qquad {\rm and}\qquad m_i(e) = \sum_{e \in E(C_i(e))}w_i'(e)\,.$$ 

\noindent
We then recall the following results:
\begin{equation}
\label{eq:wiener-tree}
W(T,w) = \sum_{e \in E(T)}n_1(e)n_2(e)\,,
\end{equation}

\begin{equation}
\label{eq:edge-wiener-tree}
\widehat{W}_e(T,w') = \sum_{e \in E(T)}m_1(e)m_2(e)\,,
\end{equation}

\begin{equation}
\label{prp:tree-weight-vertex-edge}
W_{ve}(T,w,w') = \sum_{e \in E(T)} \big(n_1(e)m_2(e) + n_2(e)m_1(e)\big).
\end{equation}

\noindent
Note that Equation \eqref{eq:wiener-tree} was proved in~\cite{klavzar-1997}, Equation ~\eqref{eq:edge-wiener-tree} in \cite{yousefi-azari-2011} and Equation \eqref{prp:tree-weight-vertex-edge} in \cite{kelenc}.
\smallskip

\noindent
The \textit{hyper-Wiener index} and the \textit{edge-hyper-Wiener index} of $G$ are defined as:
$$WW(G) = \frac{1}{4}\sum_{u \in V(G)}\sum_{v \in V(G)}d(u,v) + \frac{1}{4}\sum_{u \in V(G)}\sum_{v \in V(G)}d(u,v)^2,$$
$$WW_e(G) = \frac{1}{4}\sum_{e \in E(G)}\sum_{f \in E(G)}d(e,f) + \frac{1}{4}\sum_{e \in E(G)}\sum_{f \in E(G)}d(e,f)^2.$$
\smallskip

\noindent
Let ${\cal H}$ be the hexagonal (graphite) lattice and let $Z$ be a cricuit on it. Then a {\em benzenoid system} is induced by the vertices and edges of ${\cal H}$, lying on $Z$ and in its interior. Let $B$ be a benzenoid system. A vertex shared by three hexagons of $B$ is called an \textit{internal} vertex of $B$. A benzenoid system is said to be \textit{catacondensed} if it does not possess internal vertices. Otherwise it is called \textit{pericondensed}. Two distinct hexagons with a common edge are called \textit{adjacent}. The \textit{inner dual} of a benzenoid system $B$ is a graph which has hexagons of $B$ as vertices, two being adjacent whenever  the corresponding hexagons are also adjacent. Obviously, the inner dual of a catacondensed benzenoid system is a tree.

Let $B$ be a catacondensed benzenoid system. If we add squares between all pairs of adjacent hexagons of $B$, the obtained graph $G$ is called a \textit{phenylene}. We then say that $B$ is a \textit{hexagonal squeeze} of $G$ and denote it by $HS(G)=B$.

Let $G$ be a phenylene and $B$ a hexagonal squeeze for $B$. The edge set of $B$ can be naturally partitioned into sets $E_1'$, $E_2'$, and $E_3'$ of edges of the same direction. Denote the sets of edges of $G$ corresponding to the edges in $E_1'$, $E_2'$, and $E_3'$ by $E_1, E_2$, and $E_3$, respectively. Moreover, let $E_4 = E(G) \setminus (E_1 \cup E_2 \cup E_3)$ be the set of all the edges of $G$ that do not belong to $B$. For $i \in \lbrace 1, 2, 3, 4 \rbrace$, set $G_i = G - E_i$. The quotient graph $T_i$, $1\le i\le 4$, has connected components of $G_i$ as vertices, two such components $C$ and $C'$ being adjacent in $T_i$ if some edge in $E_i$ joins a vertex of $C$ to a vertex of $C'$. In a similar way we can define the quotient graphs $T_1', T_2', T_3'$ of hexagonal squeeze $B$. It is known \cite{chepoi-1996}  that for any benzenoid system its quotient graphs are trees.  Then a tree $T_i'$ is isomorphic to $T_i$ for $i=1,2,3$ and $T_4$ is isomorphic to the inner dual of $B$.

Now we extend the quotient trees $T_1$, $T_2$, $T_3, T_4$ to weighted trees $(T_i, w_i)$, $(T_i, w'_i)$, $(T_i, w_i, w'_i)$ as follows: 
\begin{itemize}
\item for $C \in V(T_i)$, let $w_i(C)$ be the number of edges in the component $C$ of $G_i$;
\item for $E = C_1C_2 \in E(T_i)$, let $w_i'(E)$ be the number of edges between components $C_1$ and $C_2$.
\end{itemize}

\section{A method for computing the edge-Wiener index of phenylenes}

In this section we show that the edge-Wiener index of a phenylene can be computed as the sum of Wiener indices of its weighted quotient trees. The obtained result is similar as the result for the edge-Wiener index of benzenoid systems \cite{kelenc}, but one additional quotient tree must be considered. 

\begin{theorem}
\label{thm:edge-Wiener}
Let $G$ be a phenylene. Then 
$$\widehat{W}_e(G)= \sum_{i=1}^4 \left( \widehat{W}_e(T_i, w_i') + W_v(T_i, w_i) + W_{ve}(T_i, w_i, w_i')\right).$$
\end{theorem}

\begin{proof}
Let $G$ be a phenylene and let $T_1$, $T_2$, $T_3, T_4$ be its quotient trees. For any $i \in \lbrace 1, 2, 3, 4 \rbrace$ we define $\alpha_i : E(G) \longrightarrow V(T_i) \cup E(T_i)$ by
\begin{equation}
\label{eq:alpha-i}
\alpha_i(e) = \left\{ 
   \begin{array}{l c l}
     C \in V(T_i)&; & \quad \text{$e \in E(C)$}\,,\\
     C_1C_2 \in E(T_i) &; & \quad \text{$e = ab$ and $a \in V(C_1), b \in V(C_2)$}\,.
   \end{array} \right.
\end{equation}

We will first show that 
\begin{equation} \label{osnova1}
\widehat{d}_G(e,f)= \sum_{i=1}^4 \widehat{d}_{T_i}(\alpha_i(e), \alpha_i(f))
\end{equation}
holds for any pair of edges $e,f \in E(G)$. Let $e=ab$, $f=xy$ such that $\widehat{d}_G(e,f)=\widehat{d}_G(a,x)$. Select any shortest path $P$ from $a$ to $x$ in $G$ and define $F_i = E(P) \cap E_i$ for $i \in \lbrace 1,2,3,4 \rbrace$. As $P$ is a shortest path, no two edges of $F_i$ belong to the same cut. Since $\widehat{d}_G(e,f) = |F_1| + |F_2| + |F_3| + |F_4|$ it suffices to show that for $i \in \lbrace 1,2,3,4 \rbrace$ it holds $|F_i| = \widehat{d}_{T_i}(\alpha_i(e), \alpha_i(f))$. Let $C_1, C_2 \in V(T_i)$ be connected components of $G- E_i$ such that $a \in V(C_1)$ and $x \in V(C_2)$. It follows that $\widehat{d}_{T_i}(C_1, C_2) = |F_i|$. In order to show that $\widehat{d}_{T_i}(\alpha_i(e), \alpha_i(f)) = \widehat{d}_{T_i}(C_1, C_2)$, we consider the following cases:

\begin{itemize}
\item $e \notin E_i$ and $f \notin E_i$.\\
In this case we have $\alpha_i(e) = C_1$ and $\alpha_i(f)=C_2$ and the desired conclusion is clear.
\item Exactly one of $e$ and $f$ is in $E_i$. \\
We may assume without loss of generality that $e \in E_i$ and $f \notin E_i$. Then $\alpha_i(e) = C'C_1 \in E(T_i)$ for some $C' \in V(T_i)$ and $\alpha_i(f)=C_2$. Since $\widehat{d}_{T_i}(C_1, C_2) \leq \widehat{d}_{T_i}(C', C_2)$, it follows that $\widehat{d}_{T_i}(\alpha_i(e), \alpha_i(f)) = \widehat{d}_{T_i}(C_1, C_2)$.
\item $e \in E_i$ and $f \in E_i$. \\
Now $\alpha_i(e) = C'C_1 \in E(T_i)$ for some $C' \in V(T_i)$ and $\alpha_i(f)=C''C_2$ for some $C'' \in V(T_i)$. We thus get that $\widehat{d}_{T_i}(\alpha_i(e), \alpha_i(f)) = \widehat{d}_{T_i}(C_1, C_2)$.
\end{itemize}

\noindent
Since in all possible cases $\widehat{d}_{T_i}(\alpha_i(e),\alpha_i(f)) = \widehat{d}_{T_i}(C_1, C_2)$, Equation \eqref{osnova1} holds. Applying this result we obtain
\begin{eqnarray*}
\widehat{W}_e(G) & = & \frac{1}{2} \sum_{e \in E(G)} \sum_{f \in E(G)} \widehat{d}_G(e,f) = \frac{1}{2} \sum_{e \in E(G)} \sum_{f \in E(G)} \Bigg( \sum_{i = 1}^4 \widehat{d}_{T_i}(\alpha_i(e), \alpha_i(f)) \Bigg) \\
  & = & \sum_{i=1}^4 \Bigg( \frac{1}{2} \sum_{e \in E(G)} \sum_{f \in E(G)} \widehat{d}_{T_i}(\alpha_i(e), \alpha_i(f)) \Bigg)\,.
\end{eqnarray*}
   
The obtained sums can be divided into three sums regarding the function $\alpha_i$ from Equation \eqref{eq:alpha-i}: 
\begin{eqnarray*}
\widehat{W}_e(G) & =  &   
  \sum_{i=1}^4 \Bigg( \frac{1}{2} \sum_{\substack{e \in E(G) \\ \alpha_i(e) \in E(T_i)}} \sum_{\substack{f \in E(G) \\ \alpha_i(f) \in E(T_i)}} \widehat{d}_{T_i}(\alpha_i(e), \alpha_i(f)) \\
 & & + \frac{1}{2} \sum_{\substack{e \in E(G) \\ \alpha_i(e) \in V(T_i)}} \sum_{\substack{f \in E(G) \\ \alpha_i(f) \in V(T_i)}} \widehat{d}_{T_i}(\alpha_i(e), \alpha_i(f)) \\
 &  & + \sum_{\substack{e \in E(G) \\ \alpha_i(e) \in V(T_i)}} \sum_{\substack{f \in E(G) \\ \alpha_i(f) \in E(T_i)}} \widehat{d}_{T_i}(\alpha_i(e), \alpha_i(f)) \Bigg)\,. 
\end{eqnarray*}
Application of the definition of the weighted trees finally results in 
\begin{eqnarray*}
\widehat{W}_e(G) & = & \sum_{i=1}^4 \Bigg( \frac{1}{2} \sum_{E \in E(T_i)} \sum_{F \in E(T_i)} w'_i(E)w'_i(F)\widehat{d}_{T_i}(E, F) + \\
 & & + \frac{1}{2} \sum_{C_1 \in V(T_i)} \sum_{C_2 \in V(T_i)} w_i(C_1)w_i(C_2)\widehat{d}_{T_i}(C_1, C_2)\\
 & &  + \sum_{C \in V(T_i)} \sum_{E \in E(T_i)} w_i(C)w'_i(E)\widehat{d}_{T_i}(C, E) \Bigg) \\ 
 & = & \sum_{i=1}^4 \left( \widehat{W}_e(T_i, w_i') + W_v(T_i, w_i) + W_{ve}(T_i, w_i, w_i')\right).
\end{eqnarray*}
\qed
\end{proof}

\noindent
The weighted Wiener indices of trees can be computed in linear time by using Equations \eqref{eq:wiener-tree}, \eqref{eq:edge-wiener-tree}, and \eqref{prp:tree-weight-vertex-edge}. Also, the quotient trees can be obtained in linear time as well, for the details see \cite{kelenc}. Consequently, we obtain the following corollary.

\begin{corollary}
Let $G$ be a phenylene with $m$ edges. Then the edge-Winer index of $G$ can be computed in $O(m)$ time.
\end{corollary}

\section{A method for computing the edge-hyper-Wiener index of phenylenes}

In this section, we briefly introduce the method for computing the edge-hyper-Wiener index of phenylenes, which is based on a general method for partial cubes, see \cite{tratnik1}. First, we state some important definitions.

Two edges $e_1 = u_1 v_1$ and $e_2 = u_2 v_2$ of graph $G$ are in relation $\Theta$, $e_1 \Theta e_2$, if
$$d_G(u_1,u_2) + d_G(v_1,v_2) \neq d_G(u_1,v_2) + d_G(u_1,v_2).$$
Note that this relation is also known as Djokovi\' c-Winkler relation.
The relation $\Theta$ is reflexive and symmetric, but not necessarily transitive \cite{klavzar-book}.

The {\em hypercube} $Q_n$ of dimension $n$ is defined in the following way: 
all vertices of $Q_n$ are presented as $n$-tuples $(x_1,x_2,\ldots,x_n)$ where $x_i \in \{0,1\}$ for each $1\leq i\leq n$ 
and two vertices of $Q_n$ are adjacent if the corresponding $n$-tuples differ in precisely one coordinate. A subgraph $H$ of a graph $G$ is called an \textit{isometric subgraph} if for each $u,v \in V(H)$ it holds $d_H(u,v) = d_G(u,v)$. Any isometric subgraph of a hypercube is called a {\em partial cube}. For an edge $ab$ of a graph $G$, let $W_{ab}$ be the set of vertices of $G$ that are closer to $a$ than
to $b$. We write $\langle S \rangle$ for the subgraph of $G$ induced by $S \subseteq V(G)$. 

\noindent The following theorem gives two basic characterizations of partial cubes:
\begin{theorem} \cite{klavzar-book} \label{th:partial-k} For a connected graph $G$, the following statements are equivalent:
\begin{itemize}
\item [(i)] $G$ is a partial cube.
\item [(ii)] $G$ is bipartite, and $\langle W_{ab} \rangle $ and $\langle W_{ba} \rangle$ are convex subgraphs of $G$ for all $ab \in E(G)$.
\item [(iii)] $G$ is bipartite and $\Theta = \Theta^*$.
\end{itemize}
\end{theorem}

\noindent Furthermore, it is known that when $G$ is a partial cube and $E$ is a $\Theta$-class of $G$, then $G - E$ has exactly two connected components, namely $\langle W_{ab} \rangle $ and $\langle W_{ba} \rangle$, where $ab \in E$. For more details about partial cubes see \cite{klavzar-book}.

To state the method for computing the edge-hyper-Wiener index of any partial cube, we need to introduce some more notation. If $G$ is a partial cube with $\Theta$-classes $E_1, \ldots, E_d$, we denote by $U_i$ and $U_i'$ the connected components of the graph $G - E_i$, where $i \in \lbrace 1, \ldots, d \rbrace$. For any distinct $i,j \in \lbrace 1, \ldots, d \rbrace$ set

\begin{eqnarray*}
M_{ij}^{11} & = & E(U_i) \cap E(U_j), \\
M_{ij}^{10} & = & E(U_i) \cap E(U_j'), \\
M_{ij}^{01} & = & E(U_i') \cap E(U_j), \\
M_{ij}^{00} & = & E(U_i') \cap E(U_j').
\end{eqnarray*}

\noindent
Also, for $i,j \in \lbrace 1, \ldots, d \rbrace$ and $k,l \in \lbrace 0, 1 \rbrace$ we define
$$m_{ij}^{kl} = |M_{kl}^{ij}|.$$

\begin{theorem} \cite{tratnik1} \label{osnova}
Let $G$ be a partial cube and let $d$ be the number of its $\Theta$-classes. Then
\begin{equation*}  WW_e(G) = 2W_e(G) +  \sum_{i = 1}^{d-1} \sum_{j=i+1}^d \Big( m_{ij}^{11}m_{ij}^{00} + m_{ij}^{10}m_{ij}^{01} \Big) - {{|E(G)|}\choose{2}}.
\end{equation*}
\end{theorem}

\noindent
The sum from Theorem \ref{osnova} will be denoted by $WW_e^*(G)$, i.e.
$$WW_e^*(G)= \sum_{i = 1}^{d-1} \sum_{j=i+1}^d \Big( m_{ij}^{11}m_{ij}^{00} + m_{ij}^{10}m_{ij}^{01} \Big).$$

Obviously, by case $(ii)$ of Theorem \ref{th:partial-k} any phenylene is a partial cube. An \textit{elementary cut} $C$ of a phenylene $G$ is a line segment that starts at
the center of a peripheral edge of $G$,
goes orthogonal to it and ends at the first next peripheral
edge of $G$. By $C$ we sometimes also denote the set of edges that are intersected by the corresponding elementary cut. Moreover, an elementary cut of a phenylene coincides with exactly one of its $\Theta$-classes.

Let $C_i$ and $C_j$ be two distinct elementary cuts ($\Theta$-classes) of a phenylene $G$. Since the elementary cuts can have an intersection or not, we obtain the following two options from Figure \ref{moznosti}. 

\begin{figure}[h!] 
\begin{center}
\includegraphics[scale=0.7]{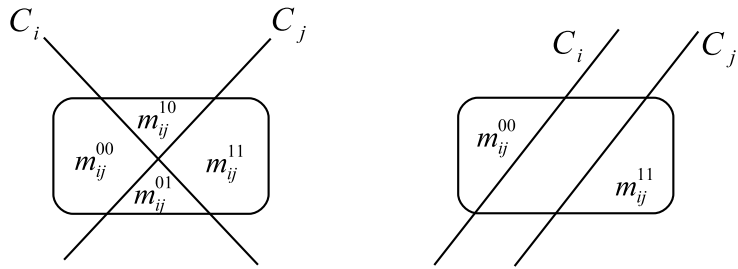}
\end{center}
\caption{\label{moznosti} Two different positions of two elementary cuts.}
\end{figure}

\noindent
In the following, the contribution of the pair $C_i, C_j$ to $WW_e^*(G)$ is denoted as $f(C_i,C_j)$. Therefore, we get

$$f(C_i,C_j)= 
\left\{ \begin{array}{lcl}
m^{00}_{ij}m^{11}_{ij} + m^{01}_{ij}m^{10}_{ij}&; & C_i \text{ and } C_j \text{ intersect}, \\
m^{00}_{ij}m^{11}_{ij}&; & \text{otherwise}.
\end{array} \right.
$$
\noindent
Hence, for phenylene $G$ with $d$ elementary cuts it holds
$$WW_e^{*}(G) = \sum_{i = 1}^{d-1} \sum_{j=i+1}^d f(C_i,C_j).$$

\section{Linear phenylenes}

A hexagon of a phenylene $PH$ is called \textit{terminal} if it has a common edge with only one square of $PH$, otherwise we say that it is \textit{internal}. If an internal hexagon has common edges with exactly two other squares, then it has exactly two vertices of degree two. If this two vertices are not adjacent, we say that such hexagon is \textit{linear}. A phenylene is called \textit{linear} if all its internal hexagons are linear. A linear phenylene with exactly $n$ hexagons will be denoted by $PH_n$, see Figure \ref{phenylene3}.

 \begin{figure}[!htb]
	\centering
		\includegraphics[scale=0.8, trim=0cm 0cm 0cm 0cm]{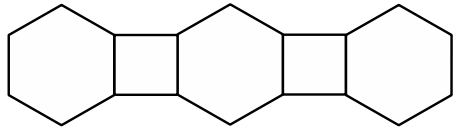}
\caption{Linear phenylene $PH_3$.}
	\label{phenylene3}
\end{figure}

We first compute the edge-Wiener index. Therefore, we determine the weighted quotient trees from Figure \ref{trees2}.

 \begin{figure}[!htb]
	\centering
		\includegraphics[scale=0.8, trim=0cm 0cm 0cm 0cm]{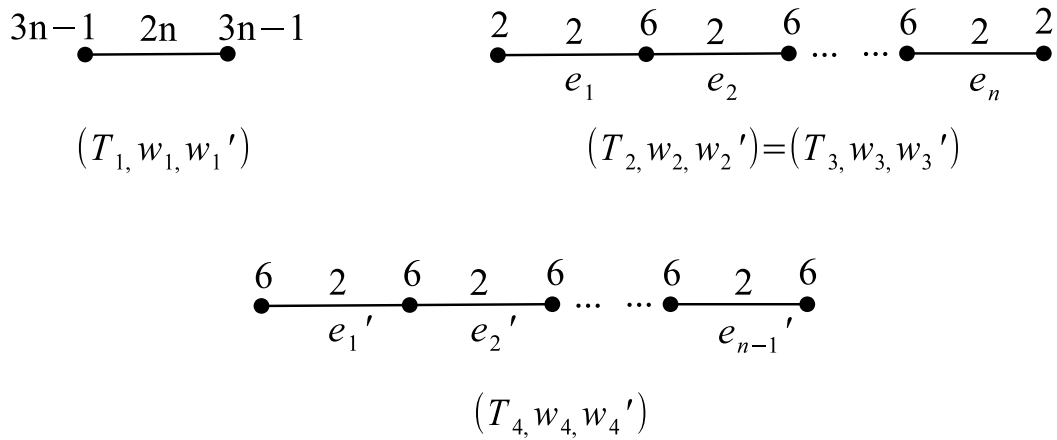}
\caption{Weighted quotient trees for the linear phenylene $PH_n$.}
	\label{trees2}
\end{figure}

\noindent
From the obtained trees it is easy to calculate the sums of weights in corresponding connected components. The results are collected in Table \ref{tab1}.

\begin{table}[H]
\centering
\begin{tabular}{|c||c|c|c|c|} \hline
  & $n_1$ & $n_2$ & $m_1$ & $m_2$ \\ \hline
 $e_i$, $1 \leq i \leq n$  & $6i-4$  &       $6(n-i)+2$  &   $2(i-1)$  &   $2(n-i)$   \\ \hline
  $e_i'$, $1 \leq i \leq n-1$    & $6i$  &       $6(n-i)$  &   $2(i-1)$  &   $2(n-i-1)$ \\
 \hline
\end{tabular}
\caption{\label{tab1} Sum of weights in the corresponding connected components of quotient trees $T_2=T_3$ and $T_4$.}
\end{table}

\noindent
Using Table \ref{tab1} we can compute the corresponding Wiener indices of $T_2$ (which are the same also for $T_3$):

\begin{eqnarray*}
W_v(T_2,w_2) &= & 6 n^3 - 6 n^2 + 4 n, \\
\widehat{W}_e (T_2,w_2') & = & \frac{1}{3} (2 n^3 - 6 n^2 + 4 n), \\
W_{ve}(T_2,w_2,w_2') & = & 4 n^3 - 8 n^2 + 4 n.
\end{eqnarray*}

\noindent
From Table \ref{tab1} we also get the corresponding Wiener indices of $T_4$:

\begin{eqnarray*}
W_v(T_4,w_4) & = & 6n^3-6n, \\
\widehat{W}_e(T_4,w_4') & = &\frac{1}{3} (2 n^3 - 12 n^2 + 22 n-12), \\
W_{ve}(T_4,w_4,w_4') & = & 4 n^3 - 12 n^2 + 8 n.
\end{eqnarray*}

\noindent
However, the computations for $T_1$ are trivial:

\begin{eqnarray*}
W_v(T_1,w_1) & = & 9 n^2 - 6 n + 1, \\
\widehat{W}_e(T_1,w_1') & = & 0, \\
W_{ve}(T_1,w_1,w_1') & = & 0.
\end{eqnarray*}

\noindent
Finally, using Theorem \ref{thm:edge-Wiener} we conclude

$$\widehat{W}_e(PH_n)=32 n^3 - 39 n^2 + 22 n - 3$$

\noindent
 and by Equation \eqref{eq:simple-connection} it follows
$$W_e(PH_n)=32 n^3 - 7 n^2 + 2 n.$$

Next, we consider the edge-hyper-Wiener index of linear phenylenes. We denote the elementary cuts of $PH_n$ with $A$, $B_1, \ldots, B_n$, $C_1, \ldots, C_n$, and $D_1, \ldots, D_{n-1}$ as shown in Figure \ref{cuts}.

 \begin{figure}[!htb]
	\centering
		\includegraphics[scale=0.8, trim=0cm 0cm 0cm 0cm]{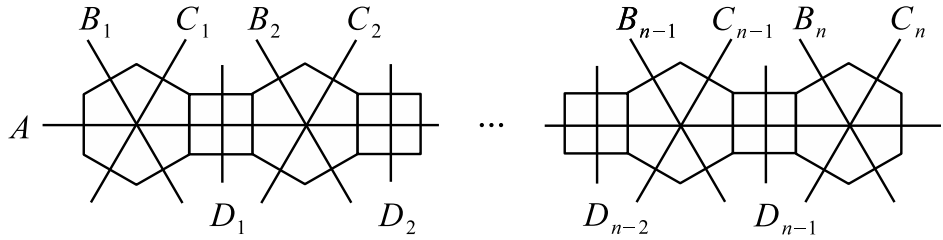}
\caption{Elementary cuts of a linear phenylene $PH_n$.}
	\label{cuts}
\end{figure}

\noindent
The contributions of all the pairs of elementary cuts are presented in Table \ref{tab2}.

\begin{table}[H]

\centering
\resizebox{\textwidth}{!}{
\begin{tabular}{|c||c|c|c|c|c|} \hline
  & $m^{00}$ & $m^{01}$ & $m^{10}$ & $m^{11}$ & $\sum f$ \\ \hline
 $A, B_i$ & $3i-3$  &       $3n-3i+1$  &   $3i-2$  &   $3n-3i$ & $3n^3 - 6n^2 + 4n$  \\ \hline
 $A, C_i$ & $3i-3$  &       $3n-3i+1$  &   $3i-2$  &   $3n-3i$ & $3n^3 - 6n^2 + 4n$  \\ \hline
  $A,D_i$ & $3i-1$  &       $3n-3i-1$  &   $3i-1$  &   $3n-3i-1$ & $3n^3 - 6n^2 + 5n - 2$ \\ \hline
 $B_i, B_j,i<j$ & $8i-6$  &    $0$  &   $8j-8i-2$  &   $8n-8j+2$ & $\frac{1}{3}(8n^4 - 32n^3 + 46n^2 - 22n) $  \\ \hline
 $C_i, C_j,i<j$ & $8i-6$  &    $0$  &   $8j-8i-2$  &   $8n-8j+2$ & $\frac{1}{3}(8n^4 - 32n^3 + 46n^2 - 22n) $  \\ \hline
 $D_i, D_j,i<j$ & $8i-2$  &    $0$  &   $8j-8i-2$  &   $8n-8j-2$ & $\frac{1}{3}(8n^4 - 32n^3 + 46n^2 - 34n + 12)$  \\ \hline
 $B_i, C_j, i<j$  & $8i-6$  &    $0$  &   $8j-8i-2$  &   $8n-8j+2$ & $\frac{1}{3}(8n^4 - 32n^3 + 46n^2 - 22n) $  \\ \hline
 $B_i, C_j, i>j$  & $8i-6$  &    $0$  &   $8j-8i-2$  &   $8n-8j+2$ & $\frac{1}{3}(8n^4 - 32n^3 + 46n^2 - 22n) $  \\ \hline
 $B_i, C_i$      & $8i-7$  &    $0$  &   $0$  &   $8n-8i+1$ & $\frac{1}{3}(32n^3 - 72n^2 + 43n) $  \\ \hline
 $B_i, D_j, i\leq j$  & $8i-6$  &    $0$  &   $8j-8i+2$  &   $8n-8j-2$ & $\frac{1}{3}(8n^4 - 16n^3 + 10n^2 - 2n) $  \\ \hline
 $B_i, D_j, i>j$  & $8j-2$  &    $0$  &   $8i-8j-6$  &   $8n-8i+2$ & $\frac{1}{3}(8n^4 - 16n^3 + 10n^2 - 2n) $  \\ \hline
 $C_i, D_j, i\leq j$  & $8i-6$  &    $0$  &   $8j-8i+2$  &   $8n-8j-2$ & $\frac{1}{3}(8n^4 - 16n^3 + 10n^2 - 2n) $  \\ \hline
 $C_i, D_j, i>j$  & $8j-2$  &    $0$  &   $8i-8j-6$  &   $8n-8i+2$ & $\frac{1}{3}(8n^4 - 16n^3 + 10n^2 - 2n) $  \\ \hline

 \hline
\end{tabular}}

\caption{\label{tab2} Contributions of pairs of elementary cuts for $PH_n$.}
\end{table}

\noindent
The expressions from Table \ref{tab2} give
$$WW_e^*(PH_n) = \frac{1}{3}(64n^4 - 133n^3 + 98n^2 - 14n - 6).$$

\noindent
Finally, by Theorem \ref{osnova} the edge-hyper Wiener index of linear phenylene $PH_n$  is equal to
$$\begin{array}{rcl}
WW_e(PH_n) & =& 2W_e(PH_n)+WW_e^{*}- {{|E(PH_n)|}\choose{2}}\\
     &&\\
                &=&  \frac{1}{3}(64n^4 +59n^3 -40n^2 +58 n - 15)\,.
\end{array}$$

\section*{Acknowledgment} 

\noindent The author acknowledge the financial support from the Slovenian Research Agency (research core funding No. P1-0297). 


\begin{thebibliography}{}
%
%

\bibitem{alizadeh-2014}
  Y.~Alizadeh, A.~Iranmanesh, T.~Do{\v{s}}li{\'c}, M.~Azari, 
  The edge Wiener index of suspensions, bottlenecks, and thorny graphs,
  {\em Glas.\ Mat.\ Ser.\ III} {\bf 49(69)} (2014) 1--12.
  
\bibitem{aroc} M. Arockiaraj, A. J. Shalini, Extended cut method for edge Wiener, Schultz and Gutman indices with applications, {\it MATCH Commun. Math. Comput. Chem.} {\bf 76} (2016) 233--250.
  
\bibitem{azari-2011}
  M.~Azari, A.~Iranmanesh, A.~Tehranian,
  A method for calculating an edge version of the Wiener number of a graph operation, 
  {\em Util.\ Math}.\ {\bf 87} (2012) 151--164. 
  
  \bibitem{buckley} F. Buckley, Mean distance in line graphs, {\it Congr. Numer.\/} {\bf 32} (1981) 153--162.
	
	\bibitem{cash} G. Cash, S. Klav\v zar, M. Petkov\v sek, Three methods for calculation of the hyper-Wiener index of molecular graphs, {\it J. Chem. Inf. Comput. Sci.} {\bf 42} (2002) 571--576.

\bibitem{chen2} A. Chen, X. Xiong, F. Lin, Explicit relation between the Wiener index and the edge-Wiener index of the catacondensed hexagonal systems, \textit{Appl. Math. Comput.} \textbf{273} (2016) 1100--1106.
  
\bibitem{chen} Y.-H. Chen, H. Wang, X.-D. Zhang, Properties of the hyper-Wiener index as a local function, {\it MATCH Commun. Math. Comput. Chem.} {\bf 76} (2016) 745--760.
  
  \bibitem{chepoi-1996}
  V.~Chepoi, 
  On distances in benzenoid systems,
  {\em J.\ Chem.\ Inf.\ Comput.\ Sci}.\ {\bf 36} (1996) 1169--1172.

\bibitem{chepoi-1997}
  V.~Chepoi, S.~Klav\v zar,
  The Wiener index and the Szeged index of benzenoid systems in linear time,
  {\em J.\ Chem.\ Inf.\ Comput.\ Sci}.\ {\bf 37} (1997) 752--755. 

\bibitem{cre-trat} M.~\v Crepnjak, N.~Tratnik, The Szeged index and the Wiener index of partial cubes with applications to chemical graphs, {\it Appl. Math. Comput.} {\bf 309} (2017) 324--333.

\bibitem{cre-trat1} M.~\v Crepnjak, N.~Tratnik, The edge-Wiener index, the Szeged indices and the PI index of benzenoid systems in sub-linear time, {\it MATCH Commun. Math. Comput. Chem.} {\bf 78} (2017) 675--688.

\bibitem{dankelmann-2009}
  P.~Dankelmann, I.~Gutman, S.~Mukwembi, H.~Swart,
  The edge-Wiener index of a graph,
  {\em Discrete Math}.\ {\bf 309} (2009) 3452--3457.
  
\bibitem{dob} A. A. Dobrynin, Hexagonal chains with segments of equal
lengths having distinct sizes and the same Wiener index, {\it MATCH Commun. Math. Comput. Chem.} {\bf 78} (2017) 121--132.

\bibitem{fur} B. Furtula, I. Gutman, \v Z. Tomovi\' c, A. Vesel, I. Pesek, Wiener-type topological indices of phenylenes, {\it Indian J. Chem.} {\bf 41A} (2002) 1767--1772.

\bibitem{gut-as} I. Gutman, A. Ashrafi, On the PI index of phenylenes and their hexagonal sqeezes, {\it MATCH Commun. Math. Comput. Chem.} {\bf 60} (2008) 135--142.

\bibitem{gu-do} I. Gutman, G. D\" om\" ot\" or, Wiener number of polyphenyls and phenylenes, {\it Z. Naturforsch} {\bf 49a} (1994) 1040--1044.

\bibitem{gutman-1997b}    
  I.~Gutman, L.~Pavlovi\'c, 
  More on distance of line graphs, 
  {\em Graph Theory Notes N.\ Y.}\ {\bf 33} (1997) 14--18.  
  
\bibitem{il-il} A. Ili\' c, M. Ili\' c, On some algorithms for computing topological indices of chemical graphs, {\it MATCH Commun. Math. Comput. Chem.} {\bf 78} (2017) 665--674.
  
  \bibitem{klavzar-book} R.~Hammack, W.~Imrich, S.~Klav\v zar, {\it Handbook of Product Graphs, Second edition},
		CRC Press, Taylor \& Francis Group, Boca Raton, 2011. 

\bibitem{iranmanesh-2015}
  A.~Iranmanesh, M.~Azari,
  Edge-Wiener descriptors in chemical graph theory: a survey,
  {\em Curr.\ Org.\ Chem}.\ {\bf 19} (2015) 219--239.

\bibitem{iranmanesh-2009}
  A.~Iranmanesh, I.~Gutman, O.~Khormali, A.~Mahmiani, 
  The edge versions of Wiener index, 
  {\em MATCH Commun.\ Math.\ Comput.\ Chem}.\ {\bf 61} (2009) 663--672.
  
  \bibitem{edge-hyper} A. Iranmanesh, A. Soltani Kafrani, O. Khormali, A new version of hyper-{W}iener index, {\it MATCH Commun. Math. Comput. Chem.\/} {\bf 65} (2011) 113--122.

\bibitem{kelenc}
A. Kelenc, S. Klav\v zar, N. Tratnik, The edge-Wiener index of benzenoid systems in linear time, \textit{MATCH Commun. Math. Comput. Chem.} \textbf{74} (2015) 521--532.

\bibitem{khalifeh-2009}
  M.~H.~Khalifeh, H.~Yousefi Azari, A.~R.~Ashrafi, S.~G.~Wagner, 
  Some new results on distance-based graph invariants, 
  {\em European J.\ Combin.}\ {\bf 30} (2009) 1149--1163.

\bibitem{klavzar-1997}
  S.~Klav\v zar, I.~Gutman,
  Wiener number of vertex-weighted graphs and a chemical application,
  {\em Discrete Appl.\ Math}.\ {\bf 80} (1997) 73--81. 
  
\bibitem{klavzar-2015}
  S.~Klav\v zar, M.~J.~Nadjafi-Arani,  
  Cut method: update on recent developments and equivalence of independent approaches,
  {\em Curr.\ Org.\ Chem}.\ {\bf 19} (2015) 348--358.  
  
\bibitem{knor-2014b}
  M.~Knor, P.~Poto\v cnik, R.~\v Skrekovski,
  Relationship between the edge-Wiener index and the Gutman index of a graph,
  {\em Discrete Appl.\ Math}.\ {\bf 167} (2014) 197--201. 
  
  \bibitem{nadjafi-arani-2012}
  M.~J.~Nadjafi-Arani, H.~Khodashenas, A.~R.~Ashrafi, 
  Relationship between edge Szeged and edge Wiener indices of graphs,
  {\em Glas.\ Mat.\ Ser.\ III} {\bf 47(67)} (2012) 21--29.
  
\bibitem{ramane} H. S. Ramane, V. V. Manjalapur, Note on the bounds on Wiener number of a graph, {\it MATCH Commun. Math. Comput. Chem.} {\bf 76} (2016) 19--22.

\bibitem{randic} M. Randi\'{c}, Novel molecular descriptor for structure-property studies, {\it Chem. Phys. Lett.\/} {\bf 211} (1993) 478--483.

\bibitem{tratnik1} N. Tratnik, A method for computing the edge-hyper-Wiener index of partial cubes and an algorithm for benzenoid systems, {\it Appl. Anal. Discr. Math.}, to appear.

\bibitem{tratnik2} N. Tratnik, The edge-Szeged index and the PI index of benzenoid systems in linear time, {\it MATCH Commun. Math. Comput. Chem.} {\bf 77} (2017) 393--406.

\bibitem{Wiener} H. Wiener, Structural determination of paraffin boiling points, {\it J. Amer. Chem. Soc.\/} {\bf 69} (1947) 17--20.

\bibitem{yousefi-azari-2011} 
  H.~Yousefi-Azari, M.~H.~Khalifeh, A.~R.~Ashrafi, 
  Calculating the edge Wiener and edge Szeged indices of graphs,
  {\em J.\ Comput.\ Appl.\ Math}.\ {\bf 235} (2011) 4866--4870. 
  
\end{thebibliography}


\end{document}